\documentclass[a4paper,11pt]{amsart}
\usepackage{amssymb,amsfonts,amsxtra,amscd,mathrsfs}
\usepackage[all]{xy}
\usepackage{fullpage}
\theoremstyle{plain}
\newtheorem{theorem}{Theorem}[section]

\newtheorem{cor}[theorem]{Corollary}
\newtheorem{prop}[theorem]{Proposition}
\theoremstyle{definition}
\newtheorem{defi}[theorem]{Definition}

\theoremstyle{remark}
\newtheorem{rem}[theorem]{Remark}
\numberwithin{equation}{section}
\newcommand{\ai}{\ensuremath{A_\infty}}
\newcommand{\ci}{\ensuremath{C_\infty}}
\newcommand{\cc}[1]{\ensuremath{\mathrm{ch}(#1)}}

\newcommand{\ext}[1]{\ensuremath{\mathrm{Ext}(#1)}}
\newcommand{\ce}[2][\bullet]{\ensuremath{C_{#1}(#2)}}
\newcommand{\hce}[2][\bullet]{\ensuremath{H_{#1}(#2)}}

\newcommand{\mcs}[1]{\ensuremath{\mathcal{MC}(#1)}}
\newcommand{\mcm}[1]{\ensuremath{\widetilde{\mathcal{MC}}(#1)}}
\newcommand{\mspc}{\ensuremath{\mathcal{M}_{g,n}}}

\newcommand{\dmcmp}{\ensuremath{\overline{\mathcal{M}}_{g,n}}}
\newcommand{\kcmp}{\ensuremath{\mathcal{K}\overline{\mathcal{M}}_{g,n}}}
\newcommand{\lcmp}{\ensuremath{\mathcal{L}\left[\dmcmp\times\Delta_{n-1}\right]}}

\newcommand{\cotimes}{\ensuremath{\hat{\otimes}}}
\newcommand{\curv}{\ensuremath{\mathcal{M}_{g,n}}}

\newcommand{\gf}{\ensuremath{\mathbb{Q}}}

\newcommand{\innprod}{\ensuremath{\langle -,- \rangle}}
\newcommand{\noproof}{\begin{flushright} \ensuremath{\square} \end{flushright}}
\DeclareMathOperator{\ad}{ad}
\DeclareMathOperator{\sgn}{sgn}

\begin{document}
\title{Classes on compactifications of the moduli space of curves through solutions to the quantum master equation}
\author{Alastair Hamilton}
\address{University of Connecticut, Mathematics Department, 196 Auditorium Road, Storrs, CT 06269. USA.}
\email{hamilton@math.uconn.edu}
\begin{abstract}
In this paper we describe a construction which produces classes in compactifications of the moduli space of curves. This construction extends a construction of Kontsevich which produces classes in the open moduli space from the initial data of a cyclic $\ai$-algebra. The initial data for our construction is what we call a `quantum $\ai$-algebra', which arises as a type of deformation of a cyclic $\ai$-algebra. The deformation theory for these structures is described explicitly. We construct a family of examples of quantum $\ai$-algebras which extend a family of cyclic $\ai$-algebras, introduced by Kontsevich, which are known to produce all the kappa classes using his construction.
\end{abstract}
\maketitle
%
%
%

\section{Introduction} \label{sec_intro}

In this paper we consider an analogue of the construction described by Kontsevich \cite{kontfeynman} which produces classes in moduli spaces of compact Riemann surfaces with unlabelled marked points. The initial data for the original construction was a cyclic $\ai$-algebra. This construction was originally described combinatorially by making use of the orbi-cell decomposition of moduli space provided by Harer \cite{harer}, Mumford, Penner \cite{penner} and Thurston. In the paper \cite{cc} it was shown that this construction could be equivalently reformulated as an instance of a more general algebraic construction of the type we consider in this paper.

In this paper we describe a construction which produces classes in a certain compactification of the moduli space of curves. This compactification was introduced by Kontsevich in his study of Witten's conjectures \cite{kontairy}. The algebraic structures which produce these classes are a type of deformation of the structure of a cyclic $\ai$-algebra, which sometimes go under the heading of `quantum $\ai$-algebras' in the physical literature \cite{herbst}, or `loop homotopy algebras' in the mathematical literature \cite{markl}. Barannikov also has a related paper in which he shows that algebras over the Feynman transform of the associative operad produce classes in a compactification of the moduli space \cite{baran}.

These quantum algebraic structures arise naturally from the description of the homology of this compactification of the moduli space that was provided by the author in \cite{ham}. In this paper it was shown that the homology of this compactification could be identified with the homology of a certain differential graded Lie algebra, which was constructed using the Lie bialgebra structures considered by various authors: \cite{chasul}, \cite{movshev}, \cite{sched}.

This differential graded Lie algebra arises as a deformation of the Lie algebra constructed by Kontsevich in \cite{kontsympgeom}, whose homology is known to recover the homology of the one point compactification of the open moduli space. As such, there is a map from this differential graded Lie algebra onto Kontsevich's given by setting the deformation parameters to zero. Geometrically, it corresponds to shrinking the boundary of this compactification to a point. Given a class on the open moduli space, we can ask whether it is possible to lift this class to the compactified moduli space and whether it is possible to parameterise the different possible liftings in some way. Using the above constructions, it is possible to interpret this as a problem in algebraic deformation theory, which we solve.

The layout of the paper is as follows. In Section \ref{sec_noncomm} we recall the definition of Kontsevich's Lie algebra of noncommutative hamiltonians and construct the differential graded Lie algebra that will be the source of the algebraic structures for our construction. In Section \ref{sec_mc} we recall basic material on Maurer-Cartan moduli spaces and describe a construction which produces classes in the homology of any differential graded Lie algebra. In Section \ref{sec_quantum} we give the definition of a quantum $\ai$-algebra as a solution to the master equation in the differential graded Lie algebra constructed in Section \ref{sec_noncomm}. Section \ref{sec_obsthy} describes the essentially standard obstruction theory that controls the process of extending a cyclic $\ai$-algebra into a quantum $\ai$-algebra. Lastly, in Section \ref{sec_examples} we consider a family of examples of quantum $\ai$-algebras which extend the family of cyclic $\ai$-algebras constructed by Kontsevich in \cite{kontfeynman}, which are known to yield all the kappa classes in the moduli space of curves using the construction in Section \ref{sec_mc}.

\subsection*{Notation and conventions}

Throughout the paper we work over the field $\gf$ of rational numbers, although we could of course work over any field of characteristic zero. By a symplectic vector space we mean a $\mathbb{Z}/2\mathbb{Z}$-graded vector space $V$ with a nondegenerate skew symmetric bilinear form $\innprod$. We restrict ourselves to the situation in which the bilinear form is \emph{odd}, except in Section \ref{sec_examples}, in which we treat some examples. This is not a necessary restriction, but one of convenience. By the parity reversion $\Pi V$ we mean the $\mathbb{Z}/2\mathbb{Z}$ graded vector space
\[\Pi V_0:=V_1\qquad\Pi V_1:=V_0.\]

If a group $G$ acts on a vector space $V$, we will denote its coinvariants by $V_G$. By a (differential) graded Lie algebra we will mean a $\mathbb{Z}/2\mathbb{Z}$-graded vector space $\mathfrak{g}$ such that $\Pi\mathfrak{g}$ has the structure of a (differential) graded Lie algebra; i.e., $\mathfrak{g}$ is equipped with an \emph{odd} Lie bracket. Here we prefer to work with $\mathbb{Z}/2\mathbb{Z}$-graded spaces as opposed to $\mathbb{Z}$-graded spaces, in order to accommodate a broader class of examples, such as those considered in Section \ref{sec_examples}.

\section{Noncommutative geometry} \label{sec_noncomm}

In this section we introduce the relevant algebraic structures that we will need for our construction. We recall the definition of Kontsevich's Lie algebra of noncommutative hamiltonians and how it can be given the structure of a Lie bialgebra. This Lie bialgebra is then used to construct, following \cite{ham}, the differential graded Lie algebra whose Maurer-Cartan elements will be the initial data for our construction. Note that in all our constructions we work formally, i.e. we work with power series rather than polynomials.

\begin{defi}
Let $V$ be a symplectic vector space. We can define a Lie algebra structure $\{-,-\}$ on
\[ \mathfrak{h}[V]:=\prod_{n=0}^\infty [(V^*)^{\otimes n}]_{\mathbb{Z}/n\mathbb{Z}} \]
by the formula
\begin{equation} \label{eqn_bracket}
\{a_1\cdots a_n,b_1\cdots b_m\} = \sum_{i=1}^n\sum_{j=1}^m (-1)^p \langle a_i,b_j \rangle^{-1} (z_{n-1}^{i-1}\cdot [a_1\cdots \hat{a_i} \cdots a_n]) (z_{m-1}^{j-1}\cdot [b_1\cdots \hat{b_j} \cdots b_m]),
\end{equation}
for $a_i,b_j\in V$; where $z_n$ denotes the $n$-cycle $(n\,n-1\ldots 2\,1)$ and $\innprod^{-1}$ denotes the dual inner product on $V^*$.
\end{defi}

In fact, we can give $\mathfrak{h}[V]$ the structure of a Lie bialgebra.

\begin{defi}
There is a Lie cobracket $\Delta:\mathfrak{h}[V]\to\mathfrak{h}[V]\otimes\mathfrak{h}[V]$ defined by the formula
\begin{equation} \label{eqn_cobracket}
\Delta(a_1\cdots a_n) = \frac{1}{2}\sum_{i<j}(-1)^p\langle a_i,a_j \rangle^{-1}[1+(1 \, 2)]\cdot[ (a_{i+1}\cdots a_{j-1}) \otimes (a_{j+1}\cdots a_n a_1\cdots a_{i-1})];
\end{equation}
where
\begin{displaymath}
\begin{split}
p:= & |a_i|(|a_1|+\ldots+|a_i|) + |a_j|(|a_1|+\ldots+|a_j|) \\
& + (|a_1|+\ldots+|a_{i-1}|)(|a_{i+1}|+\ldots+|a_{j-1}|+|a_{j+1}|+\ldots+|a_n|).
\end{split}
\end{displaymath}
\end{defi}

The following proposition can be checked by a series of direct calculations.

\begin{prop}
The bracket $\{-,-\}$ and the cobracket $\Delta$ defined above give $\mathfrak{h}[V]$ the structure of an involutive Lie bialgebra. This means that they satisfy the following compatibility conditions:
\begin{enumerate}
\item $\Delta([x,y])  = [x,\Delta(y)]-(-1)^{xy}[y,\Delta(x)]$,
\item $[-,-]\circ\Delta  = 0.$
\end{enumerate}
\end{prop}
\noproof

\begin{rem}
The Lie bialgebra $\mathfrak{h}[V]$ has a natural grading determined by the length of a cyclic word. We denote by
\[\mathfrak{h}_{\geq i}[V]:=\prod_{n=i}^\infty [(V^*)^{\otimes n}]_{\mathbb{Z}/n\mathbb{Z}}\]
the subspace of $\mathfrak{h}[V]$ consisting of words of length $\geq i$. Note that $\mathfrak{h}_{\geq 2}[V]$ is a Lie subalgebra of $\mathfrak{h}[V]$. It is this Lie algebra $\mathfrak{h}_{\geq 2}$ which is the subject of Kontsevich's original theorem \cite{kontsympgeom}, which links its (stable) homology to the homology of the one point compactification of the moduli space of curves.
\end{rem}

There is a standard way to construct a differential graded Lie algebra from any involutive Lie bialgebra, which we now recall.

\begin{defi}
Let $\mathfrak{g}$ be a differential graded Lie algebra. The Chevalley-Eilenberg complex of $\mathfrak{g}$, denoted by $\ce{\mathfrak{g}}$, is the complex whose underlying vector space is the completed symmetric algebra on $\mathfrak{g}$:
\[ \ce{\mathfrak{g}}:=\widehat{S}(\mathfrak{g})=\prod_{n=0}^{\infty} (\mathfrak{g}^{\cotimes n})_{S_n}. \]
The differential $\delta:\ce{\mathfrak{g}}\to\ce{\mathfrak{g}}$ is defined by the formula,
\begin{displaymath}
\begin{split}
\delta(g_1\cdots g_n):=&\sum_{1\leq i<j\leq n} (-1)^p\{g_i,g_j\} \, g_1\cdots \hat{g_i} \cdots \hat{g_j} \cdots g_n \\
& + \sum_{1\leq i\leq n} (-1)^q d(g_i) \, g_1\cdots \hat{g_i} \cdots  g_n;
\end{split}
\end{displaymath}
where
\begin{displaymath}
\begin{split}
p:= & |g_i|(|g_1|+\ldots+|g_{i-1}|) + |g_j|(|g_1|+\ldots+|g_{j-1}|) + |g_i||g_j|, \\
q:= & |g_i|(|g_1|+\ldots+|g_{i-1}|)
\end{split}
\end{displaymath}
and $d$ is the differential on $\mathfrak{g}$. The homology of this complex is known as the Chevalley-Eilenberg homology of the differential graded Lie algebra $\mathfrak{g}$ and is denoted by $\hce{\mathfrak{g}}$.
\end{defi}

\begin{rem}
Note that $\delta$ is a coderivation of the canonical coproduct on the symmetric algebra, so that in fact $\hce{\mathfrak{g}}$ has the structure of a coalgebra.
\end{rem}

There is a natural grading on $\widehat{S}(\mathfrak{g})$ determined by counting the order of the tensors. We will denote by
\[\widehat{S}^{\geq i}(\mathfrak{g}):=\prod_{n=i}^{\infty} (\mathfrak{g}^{\cotimes n})_{S_n}\]
the subspace of $\widehat{S}(\mathfrak{g})$ consisting of tensors of order $\geq i$.

If $\mathfrak{g}$ is a Lie bialgebra, then we may define an additional differential $\Delta$ on the Chevalley-Eilenberg complex of $\mathfrak{g}$ by extending the Lie cobracket on $\mathfrak{g}$ using the Leibnitz rule:
\[\Delta(g_1\cdots g_n):=\sum_{i=1}^n (-1)^p\Delta(g_i) \, g_1\cdots \hat{g_i} \cdots g_n,\]
where $p:=|h_i|(|h_1|+\ldots+|h_{i-1}|)$. The condition $\Delta^2=0$ follows from the coJacobi identity. The Lie bracket $[-,-]$ on $\mathfrak{g}$ extends to a Lie bracket on the whole of $\ce{\mathfrak{g}}$ using the Leibniz rule, giving $\ce{\mathfrak{g}}$ the structure of a differential graded Lie algebra with respect to both $\Delta$ and $\delta$.

Consider now the involutive Lie bialgebra $\mathfrak{h}$. We can define a differential graded Lie algebra out of $\mathfrak{h}$ by taking the Chevalley-Eilenberg complex of $\mathfrak{h}$ and equipping it with a deformed differential. The underlying space is
\[ \mathfrak{l}:=\ce{\mathfrak{h}}\cotimes\gf[[\gamma]]\]
and the differential is $d:=\gamma\cdot\delta+\Delta$. Technically, this is not the object we want, so we now describe how it should be modified. First we note that $\mathfrak{h}$ splits as
\[ \mathfrak{h}=\mathfrak{h}_{\geq 1}\times\gf. \]
So, by identifying the completed symmetric algebra on the field $\gf$ with the free power series algebra in one variable $\nu$ we see that
\[ \mathfrak{l}= \widehat{S}(\mathfrak{h}_{\geq 1})\cotimes\gf[[\gamma,\nu]]. \]
Now we notice that the summand corresponding to $\gf[[\gamma,\nu]]$ in $\mathfrak{l}$ is a differential graded ideal of $\mathfrak{l}$ and therefore
\[ \mathfrak{l}':= \widehat{S}^{\geq 1}(\mathfrak{h}_{\geq 1})\cotimes\gf[[\gamma,\nu]] \]
inherits the structure of a differential graded Lie algebra. Finally, we also need to throw away the summand corresponding to $V$. That is, $\mathfrak{l}'$ has a grading where the order of each deformation parameter is 1 and a cyclic word in $\mathfrak{h}$ of length $i$ has order $i$. We throw away the part of order one; that is, we consider the differential graded Lie subalgebra generated by terms of order $\geq 2$. This yields the desired differential graded Lie algebra, which we denote by $\Lambda_{\gamma,\nu}[V]$. There are maps
\begin{equation} \label{eqn_dgmap}
\Lambda_{\gamma,\nu}[V]\to\Lambda_{\gamma}[V]\to\mathfrak{h}_{\geq 2}[V]
\end{equation}
of differential graded Lie algebras, where $\mathfrak{h}_{\geq 2}[V]$ has the trivial differential. Here $\Lambda_{\gamma}[V]$ denotes the subspace of $\widehat{S}^{\geq 1}(\mathfrak{h}_{\geq 1})\cotimes\gf[\gamma]$ generated by terms of order $\geq 2$. It carries the natural structure of a differential graded Lie algebra, which is induced by stipulating that the map $\Lambda_{\gamma,\nu}[V]\to\Lambda_{\gamma}[V]$, determined by setting the deformation parameter $\nu$ to zero, is a map of differential graded Lie algebras. The right-hand map of \eqref{eqn_dgmap} is given by setting the other deformation parameter $\gamma$ to zero and projecting onto the summand
\[\mathfrak{h}_{\geq 2}=\widehat{S}^1(\mathfrak{h}_{\geq 2})\subset\widehat{S}^{\geq 1}(\mathfrak{h}_{\geq 1}).\]

The main result of \cite{ham} is that the homology of the differential graded Lie algebra $\Lambda_{\gamma,\nu}[V]$, after stabilising with respect to the dimension of $V$, recovers precisely the homology of a certain compactification of the decorated moduli space of curves with unlabelled marked points. The maps \eqref{eqn_dgmap} described above correspond geometrically to shrinking parts of the boundary of this compactification to a point.

There are two relevant compactifcations, defined by introducing an equivalence relation on the Deligne-Mumford compactification. In the first, two stable curves are identified as being equivalent if they are homeomorphic through a map which preserves the complex structure on the irreducible components which \emph{contain at least one marked point}. We denote this compactification by $\kcmp$. In the second, we consider \emph{decorated} stable curves, in which each marked point is assigned a non-negative real number, called a \emph{perimeter}, which is allowed to vanish; two decorated stable curves are identified as being equivalent if they are homeomorphic through a map which preserves the complex structure on the irreducible components \emph{containing at least one marked point with positive perimeter.} We denote this compactification by $\lcmp$. See \cite{kontairy}, \cite{looi}, \cite{mondello} and \cite{ham} for details.

Summarising the results of \cite{ham}, we have the following commutative diagram, where $pt$ denotes the one point compactification and the left-hand maps are defined by shrinking the relevant parts of the boundary to a point. On the right-hand side we consider \emph{stable} Chevalley-Eilenberg homology, i.e. we consider the stable limit of the homology groups as the dimension of $V$ tends to infinity.
\begin{equation} \label{eqn_shrink}
\xymatrix{H_{\bullet} \left(\bigsqcup_{\begin{subarray}{c} n\geq 1 \\ g>1-\frac{n}{2} \end{subarray}} \lcmp\right) \ar@{=}[r]  \ar[d] & \hce[\bullet]{\Lambda_{\gamma,\nu}} \ar[d] \\ H_{\bullet}\left(\bigsqcup_{\begin{subarray}{c} n\geq 1 \\ g>1-\frac{n}{2} \end{subarray}} [\kcmp \times \Delta^{\circ}_{n-1}]^{pt}\right) \ar@{=}[r] \ar[d] & \hce[\bullet]{\Lambda_\gamma} \ar[d] \\  H_{\bullet} \left(\bigsqcup_{\begin{subarray}{c} n\geq 1 \\ g>1-\frac{n}{2} \end{subarray}} [\mspc \times \Delta^{\circ}_{n-1}]^{pt}\right) \ar@{=}[r] & \hce[\bullet]{\mathfrak{h_{\geq 2}}} }
\end{equation}

\section{Maurer-Cartan moduli spaces and characteristic classes} \label{sec_mc}

In this section we recall basic material on Maurer-Cartan moduli spaces of differential graded Lie algebras and explain how to construct classes in the homology of any differential graded Lie algebra by exponentiating elements in the associated Maurer-Cartan moduli space. An important technical point is that this differential graded Lie algebra should be pronilpotent.

\subsection{Maurer-Cartan moduli spaces}

\begin{defi}
Let $\mathfrak{g}$ be a pronilpotent differential graded Lie algebra. The Maurer-Cartan set $\mcs{\mathfrak{g}}$ of this differential graded Lie algebra is defined as
\[ \mcs{\mathfrak{g}}:=\left\{x\in\mathfrak{g}_0:dx+\frac{1}{2}[x,x]=0\right\}.\]
The Maurer-Cartan moduli space $\mcm{\mathfrak{g}}$ of the differential graded Lie algebra is defined as the quotient of the above set by the following action of $\mathfrak{g}_1$:
\[\exp(y)\cdot x:=x+\sum_{n=0}^\infty\frac{1}{(n+1)!}[\ad y]^n(dy+[y,x]).\]
This action is well defined as the differential graded Lie algebra is pronilpotent.
\end{defi}

The above action is defined in such a way that
\begin{equation}\label{eqn_mcflow}
\frac{d}{dt}\left[\exp(ty)\cdot x\right]=dy+\ad y[\exp(ty)\cdot x].
\end{equation}

The Maurer-Cartan set is a functorial construction. That is to say that a morphism $\mathfrak{g}\to\mathfrak{g}'$ of differential graded Lie algebras induces a map $\mcs{\mathfrak{g}}\to\mcs{\mathfrak{g}'}$. Furthermore, one can check that it descends to a map on the Maurer-Cartan moduli spaces $\mcm{\mathfrak{g}}$ and $\mcm{\mathfrak{g}'}$.

\subsection{Characteristic classes}

In this section we explain a construction, well-known to experts (cf. \cite{ss}), which produces classes in the homology of any differential graded Lie algebra by exponentiating elements in the Maurer-Cartan moduli space.

\begin{defi}
Let $\mathfrak{g}$ be a differential graded Lie algebra and take an element $x\in\mcs{\mathfrak{g}}$. The characteristic class $\cc{x}\in\ce{\mathfrak{g}}$ is defined by the formula
\[\cc{x}:=\exp(x)=1+x+\frac{1}{2}x\cdot x +\cdots + \frac{1}{n!}x^n+\cdots.\]
Note that this is an inhomogeneous element in the Chevalley-Eilenberg complex.
\end{defi}

\begin{theorem} \
\begin{enumerate}
\item
$\cc{x}$ is a cycle.
\item
Maurer-Cartan equivalent elements produce homologous classes.
\end{enumerate}
Hence, the characteristic class is a map
\[\mathrm{ch}:\mcm{\mathfrak{g}}\to\hce{\mathfrak{g}}.\]
\end{theorem}

\begin{proof} \
\begin{enumerate}
\item
This follows by direct calculation
\begin{displaymath}
\begin{split}
\delta(\cc{x})&=\sum_{n=0}^\infty\frac{1}{n!}\left(ndx\cdot x^{n-1}+\frac{n(n-1)}{2}[x,x]\cdot x^{n-2}\right).\\
&= \left(dx+\frac{1}{2}[x,x]\right)\cdot\cc{x}.
\end{split}
\end{displaymath}
Hence $\cc{x}$ is a cycle if and only if $x$ satisfies the Maurer-Cartan equation.
\item
Suppose we have two Maurer-Cartan equivalent elements $x,x'\in\mcm{\mathfrak{g}}$ so that there is an element $y\in\mathfrak{g}_1$ such that
\[x'=\exp(y)\cdot x.\]
Consider the map $\gamma_y:\ce{\mathfrak{g}}\to\ce{\mathfrak{g}}$ defined by the formula
\[\gamma_y(g_1\cdots g_k):= dy\cdot g_1\cdots g_k + \sum_{i=1}^k (-1)^p [y,g_i]\cdot g_1\cdots \hat{g}_i\cdots g_k,\]
where $p:= |g_i|(|g_1|+\ldots+|g_{i-1}|)$. This map is nullhomotopic. A contracting homotopy is provided by the map $s_y:\ce{\mathfrak{g}}\to\ce{\mathfrak{g}}$ defined by the formula
\[s_y(g_1\cdots g_k):=y\cdot g_1\cdots g_k.\]
So we have the identity $\gamma_y=\delta\circ s_y + s_y\circ\delta$.

Now notice that it follows from Equation \eqref{eqn_mcflow} that the curves $c_1(t)$ and $c_2(t)$ defined by the formulae
\begin{displaymath}
\begin{split}
c_1(t)&:=\cc{\exp(ty)\cdot x},\\
c_2(t)&:=\exp(t\gamma_y)[\cc{x}];
\end{split}
\end{displaymath}
are both solutions to the initial value problem
\[\frac{dc}{dt}=\gamma_y(c(t)),\quad c(0)=\cc{x}.\]
Hence they must coincide. Taking $t=1$ we conclude
\[\cc{x'}=\exp(\gamma_y)[\cc{x}].\]
Since the map $\gamma_y$ is nullhomotopic, it follows that $x$ and $x'$ give rise to homologous classes.
\end{enumerate}
\end{proof}

The characteristic class map is a natural transformation. This means that if $\mathfrak{g}\to\mathfrak{g}'$ is a morphism of differential graded Lie algebras then we have a commutative diagram
\[\xymatrix{\mcm{\mathfrak{g}} \ar[d] \ar[r]^{\mathrm{ch}}& \hce{\mathfrak{g}} \ar[d] \\ \mcm{\mathfrak{g}'} \ar[r]^{\mathrm{ch}} & \hce{\mathfrak{g}'} }\]

\section{Quantum A-infinity algebras} \label{sec_quantum}

In this section we give the definition of a quantum $\ai$-algebra as the solution to a certain master equation and explain how any quantum $\ai$-algebra gives rise to a cyclic $\ai$-algebra. Furthermore, we explain how one can produce quantum $\ai$-algebras from cyclic $\ai$-algebras by imposing on them a kind of quantum constraint. Lastly, we describe by means of the construction outlined in Section \ref{sec_mc} how quantum $\ai$-algebras produce classes in a compactification of the moduli space of curves. These classes lift the classes in the open moduli space which are defined by the classical (cyclic) part of the quantum $\ai$-structure via Kontsevich's construction \cite{kontfeynman}.

\subsection{$\ai$-structures}

\begin{defi}
Let $W$ be a vector space with a symmetric nondegenerate odd bilinear form $\innprod$, so that $V:=\Pi W$ is a symplectic vector space. A cyclic $\ai$-structure on $W$ is an element $h\in\mcs{\mathfrak{h}_{\geq 2}[V]}$, that is an element satisfying the classical master equation
\[\{h,h\}=0.\]
We shall denote the corresponding cyclic $\ai$-algebra by $A:=(W,\innprod,h)$.
\end{defi}

\begin{rem}
This definition comes from the perspective that a cyclic $\ai$-structure on $W$ is equivalent to a symplectic homological vector field on the formal noncommutative symplectic supermanifold $\widehat{T}(V^*)$; cf. \cite{kontfeynman}, \cite{cinf} and \cite[\S 10]{noncom}.
\end{rem}

There is a convenient definition of the cyclic Hochschild cohomology of a cyclic $\ai$-algebra, cf. \cite{noncom}, that will be required in Section \ref{sec_obsthy} when we come to deal with the obstruction theory for quantum $\ai$-algebras.

\begin{defi} \label{def_cychom}
Let $A:=(W,\innprod,h)$ be a cyclic $\ai$-algebra. The cyclic Hochschild complex of $A$ is the complex
\[CC^\bullet(A):=\left(\mathfrak{h}_{\geq 1}[V],\ad(h)\right).\]
The Maurer-Cartan condition on $h$ ensures that the map $\ad(h)$ squares to zero. The cohomology of this complex is the cyclic Hochschild cohomology of $A$ and is denoted by $HC^\bullet(A)$.
\end{defi}

To define a quantum $\ai$-algebra, we look for solutions to the master equation in $\Lambda_{\gamma,\nu}[V]$.

\begin{defi}
A quantum $\ai$-structure on $W$ is an element $h_{\gamma,\nu}\in\mcs{\Lambda_{\gamma,\nu}[V]}$, that is an element satisfying the quantum master equation
\[d(h_{\gamma,\nu})+\frac{1}{2}\{h_{\gamma,\nu},h_{\gamma,\nu}\}=0.\]
\end{defi}

Now recall \eqref{eqn_dgmap} that there is a map $\Lambda_{\gamma,\nu}[V]\to\mathfrak{h}_{\geq 2}[V]$ of differential graded Lie algebras defined by setting the deformation parameters to zero and projecting onto the relevant summand. This map induces a map on the associated Maurer-Cartan sets, hence any quantum $\ai$-structure on $W$ yields a cyclic $\ai$-structure on $W$ by simply taking its image under this projection.

In fact, it is possible to go backwards. By imposing a certain constraint on a cyclic $\ai$-structure, we can produce a solution to the quantum master equation. Consider the Lie subalgebra $\mathfrak{k}[V]$ of $\mathfrak{h}_{\geq 2}[V]$ consisting of elements $h$ satisfying
\begin{equation} \label{eqn_quantumconstraint}
\Delta(h) = 0.
\end{equation}
That this is a Lie subalgebra follows from the compatibility condition between the cobracket and the bracket in our Lie bialgebra structure. This Lie subalgebra fits into the following diagram of differential graded Lie algebras
\[\xymatrix{ & \Lambda_{\gamma,\nu}[V] \ar@{->>}[d] \\ \mathfrak{k}[V] \ar@{^{(}->}[ru] \ar@{^{(}->}[r] & \mathfrak{h}_{\geq 2}[V] }, \]
where the diagonal inclusion embeds $\mathfrak{k}[V]$ into the summand of $\Lambda_{\gamma,\nu}[V]$ corresponding to $\mathfrak{h}_{\geq 2}[V]$. This diagram gives rise to a corresponding diagram of Maurer-Cartan sets; hence any Maurer-Cartan element in $\mathfrak{k}[V]$, that is, an even element $h\in\mathfrak{h}_{\geq 2}[V]$ satisfying
\[ \Delta(h)=0\quad\text{and}\quad\{h,h\}=0,\]
yields a quantum $\ai$-structure on $W$.

\subsection{Classes in the moduli space of curves}

The algebraic structures of the preceding section can be used to produce classes in the moduli space of curves by making use of Theorem 1.1 of \cite{kontsympgeom} and Theorem 5.6 of \cite{ham}. Recall diagram \eqref{eqn_shrink}, which links the homology of the differential graded Lie algebras constructed in Section \ref{sec_noncomm} to the homology of the corresponding compactifications of the moduli space. Applying the characteristic class map yields the following commutative diagram:

\begin{equation} \label{eqn_liftclass}
\xymatrix{ \mcs{\Lambda_{\gamma,\nu}[V]} \ar[r]^{\mathrm{ch}} \ar[d] & H_\bullet(\Lambda_{\gamma,\nu}) \ar@{=}[r] \ar[d] & H_\bullet\left(\sqcup_{g,n}\lcmp\right) \ar[d] \\ \mcs{\Lambda_{\gamma}[V]} \ar[r]^{\mathrm{ch}} \ar[d] & H_\bullet(\Lambda_{\gamma}) \ar@{=}[r] \ar[d] & H_\bullet\left(\sqcup_{g,n}[\kcmp\times\Delta_{n-1}^{\circ}]^{pt}\right) \ar[d] \\ \mcs{\mathfrak{h}_{\geq 2}[V]} \ar[r]^{\mathrm{ch}} & H_\bullet(\mathfrak{h}_{\geq 2}) \ar@{=}[r] & H_\bullet\left(\sqcup_{g,n}[\curv\times\Delta^{\circ}_{n-1}]^{pt}\right)}
\end{equation}

Hence we see that the problem of lifting a homology class defined on $\curv^{pt}$ to a class on the compactification $\lcmp$ can be interpreted as the problem of lifting a cyclic $\ai$-structure to a quantum $\ai$-structure. Likewise, the problem of lifting a class on $\curv^{pt}$ to a class on $\kcmp$ is analogous to the problem of lifting a solution of the master equation in $\mathfrak{h}_{\geq 2}[V]$ to a solution of the master equation in $\Lambda_{\gamma}[V]$. This process is controlled by a certain deformation theory, which is the subject of the next section.

\section{Obstruction theory} \label{sec_obsthy}

In this section we will describe the deformation theory which controls the process of lifting a solution to the master equation in $\mathfrak{h}_{\geq 2}$ to a solution of the master equation in $\Lambda_{\gamma,\nu}$. This is an iterative procedure in which each step of the process is controlled by a certain cohomology theory, which in this case coincides with the cyclic Hochschild cohomology of the underlying $\ai$-structure. To describe this iterative process, we fix a symplectic vector space $V:=\Pi W$ and provide $\Lambda_{\gamma,\nu}:=\Lambda_{\gamma,\nu}[V]$ with a filtration. We define the deformation parameter $\gamma$ to have order 2, whilst the deformation parameter $\nu$ and an element $h\in\mathfrak{h}_{\geq 2}$ have order 1. Our convention is then that the tensor
\[\gamma^g\nu^n h_1\cdots h_k\]
in $\Lambda_{\gamma,\nu}$ has homogeneous order $2g+n+k-1$. The filtration $F_p$ then comes from this ordering:
\[ F_p[\Lambda_{\gamma,\nu}]=\text{all tensors of order}\geq p. \]

This is a complete filtration of $\Lambda_{\gamma,\nu}$ in which the bracket preserves the filtration degree and the differential increases the filtration degree by 1:
\[ [F_p,F_q]\subset F_{p+q} \qquad d(F_p)\subset F_{p+1}. \]

The set up of the problem is as follows. We start with a fixed element $h\in\mcs{\mathfrak{h}_{\geq 2}[V]}$, which gives $W$ the structure of a cyclic $\ai$-algebra $A:=(W,\innprod,h)$. We want to extend this to an element $h_{\gamma,\nu}\in\mcs{\Lambda_{\gamma,\nu}}$; that is we are interested in describing the fiber $\pi^{-1}{h}$ of quantum $\ai$-structures over a fixed cyclic $\ai$-structure $h$. We know that the subspace $\Lambda^{\mathrm{odd}}_{\gamma,\nu}$ of $\Lambda_{\gamma,\nu}$ generated by odd elements acts on $\mcs{\Lambda_{\gamma,\nu}}$, but this action does not fix $h$, so we instead consider the action of $F_1[\Lambda^{\mathrm{odd}}_{\gamma,\nu}]$, which does. We denote the corresponding moduli space given by quotienting out by this action by $\widetilde{\pi^{-1}{h}}$.

Note that each $F_i$ is a differential graded ideal, hence each quotient
\[ \Lambda_{\gamma,\nu}/F_i[\Lambda_{\gamma,\nu}]\]
has the structure of a differential graded Lie algebra. The obstruction theory is the theory which describes the process of extending the element $h$ up this tower of differential graded Lie algebras.
\[h\in\mathfrak{h}_{\geq 2}=\Lambda_{\gamma,\nu}/F_1\leftarrow\Lambda_{\gamma,\nu}/F_2\leftarrow\ldots\Lambda_{\gamma,\nu}/F_i\leftarrow\ldots.\]

Given any element $x_n\in\mcs{\Lambda_{\gamma,\nu}/F_{n+1}}$ we can write it as
\[ x_n= h_0+\cdots+h_n \]
where $h_i$ is an element of order $i$ and we assume that $h_0=h$. We are interested in lifting it to an element $x_{n+1}\in\mcs{\Lambda_{\gamma,\nu}/F_{n+2}}$. There is an obstruction to this which is defined as follows. Since $x_n\in\mcs{\Lambda_{\gamma,\nu}/F_{n+1}}$ there is an element $o_{n+1}$ of order $n+1$ defined by the equation
\[ d(x_n)+\frac{1}{2}[x_n,x_n] = o_{n+1} \mod F_{n+2}. \]

Since $h\in\mcs{\mathfrak{h}_{\geq 2}}$ and $\mathfrak{h}_{\geq 2}$ is a Lie subalgebra of $\Lambda_{\gamma,\nu}$, the adjoint map $\ad h:=[h,-]$ defines a differential on $\Lambda_{\gamma,\nu}$. We denote the corresponding cohomology theory by $\mathcal{H}^\bullet(A)$. Since $h$ has order zero, the map $\ad h$ preserves the filtration and $\mathcal{H}^\bullet(A)$ is naturally graded by the filtration degree. The following argument is standard in deformation theory; see for example \cite{murray} or \cite{jim}.

\begin{theorem}
The element $x_n$ is extendable to an element $x_{n+1}$ if and only if the obstruction $o_{n+1}$, which is a cocyle, vanishes as a cohomology class.
\end{theorem}

\begin{proof}
Let us prove that $o_{n+1}$ is a cocycle. This follows from the following calculation.
\begin{displaymath}
\begin{split}
[h,o_{n+1}] &= [h,d(x_n)]+\frac{1}{2}[h,[x_n,x_n]] \mod F_{n+2} \\
&= [dh,x_n]-d([h,x_n]) +\frac{1}{2}([[h,x_n],x_n]-[x_n,[h,x_n]]) \mod F_{n+2} \\
\end{split}
\end{displaymath}
Now write $x_n=h+\bar{x}_n$. Since $d(x_n)+\frac{1}{2}[x_n,x_n] = 0 \mod F_{n+1}$ we get
\begin{displaymath}
\begin{split}
& dh+d\bar{x}_n + \frac{1}{2}[\bar{x}_n,\bar{x}_n] + [\bar{x}_n,h] = 0 \mod F_{n+1} \\
& d([h,\bar{x}_n]) = -\frac{1}{2}d([\bar{x}_n,\bar{x}_n]) \mod F_{n+2} \\
\end{split}
\end{displaymath}
Combining this yields
\begin{displaymath}
\begin{split}
[h,o_{n+1}] &= [dh,h] + [dh,\bar{x}_n] - d([h,\bar{x}_n]) + [[h,\bar{x}_n],\bar{x}_n] \mod F_{n+2} \\
&= [dh,h] + [dh,\bar{x}_n] + \frac{1}{2}d([\bar{x}_n,\bar{x}_n]) - \frac{1}{2}[[\bar{x}_n,\bar{x}_n],\bar{x}_n] - [d(\bar{x}_n),\bar{x}_n] - [dh,\bar{x}_n] \mod F_{n+2} \\
&= 0 \mod F_{n+2} \\
\end{split}
\end{displaymath}
so $o_{n+1}$ is a cocycle.

Suppose that $x_n$ is extendable. This happens if and only if there exists an element $h_{n+1}\in\Lambda_{\gamma,\nu}$ of order $n+1$ such that $x_{n+1}:=x_n+h_{n+1}$ satisfies the quantum master equation
\[ d(x_{n+1})+\frac{1}{2}[x_{n+1},x_{n+1}] = 0 \mod F_{n+2} \]
which is equivalent to
\begin{equation} \label{eqn_obstruction}
o_{n+1} +[h_{n+1},h] = 0.
\end{equation}
Hence $x_n$ is extendable if and only if $o_{n+1}$ is a coboundary.
\end{proof}

One aspect of the cohomology theory, which we have just seen, is that it describes the obstruction to extending a solution of the quantum master equation modulo $F_n$ to the next level. The other aspect of the cohomology theory is that it parameterises the different kind of extensions that we can make.

Suppose that $h_{n+1}$ and $h'_{n+1}$ are two extensions of $x_n$. Since an element is an extension if and only if Equation \eqref{eqn_obstruction} holds, it follows that any two extensions differ by a cocycle. We say that the extensions $h_{n+1}$ and $h'_{n+1}$ are equivalent if there exists an element $\xi_{n+1}\in\Lambda_{\gamma,\nu}$ of order $n+1$ such that
\[\exp(\xi_{n+1})\cdot(x_n+h_{n+1})=x_n+h'_{n+1} \mod F_{n+2}. \]
This happens if and only if $h_{n+1}+[\xi_{n+1},h]=h'_{n+1}$, i.e. these two extensions differ by a coboundary. We denote the set of equivalence classes of extensions of $x_n$ by $\ext{x_n}$. This leads to the following standard theorem.

\begin{theorem} \label{thm_cohomact}
The cohomology group $\mathcal{H}^{n+1}(A)$ acts freely and transitively on equivalence classes of extensions of $x_n$:
\begin{displaymath}
\begin{array}{ccc}
\mathcal{H}^{n+1}(A)\times\ext{x_n} & \to & \ext{x_n} \\
(\alpha_{n+1},h_{n+1}) & \mapsto & \alpha_{n+1} + h_{n+1} \\
\end{array}
\end{displaymath}
\end{theorem}
\noproof

Now we will identify the cohomology theory $\mathcal{H}^{\bullet}(A)$ as something more familiar. The underlying complex is
\[\widehat{S}(\mathfrak{h}_{\geq 1})\cotimes\gf[[\gamma,\nu]]\]
where the differential comes from extending the differential $\ad(h)$ on $\mathfrak{h}_{\geq 1}$ using the Leibniz rule. It therefore follows from Definition \ref{def_cychom} and the Kunneth theorem that the cohomology of this complex is
\[\mathcal{H}^{\bullet}(A)=\widehat{S}(HC^{\bullet}(A))\cotimes\gf[[\gamma,\nu]],\]
where $HC^{\bullet}(A)$ denotes the cyclic Hochschild cohomology of $A$.

\begin{cor} \label{cor_vanish}
Let $A$ be any cyclic $\ai$-algebra whose cyclic Hochschild cohomology vanishes in odd degrees, then $A$ admits an extension to a quantum $\ai$-algebra. Furthermore, if $HC^\bullet(A)$ vanishes entirely, then any such extension is unique up to isomorphism.
\end{cor}

\begin{proof}
Since all the obstructions lie in odd degrees, it follows that if $HC^\bullet(A)$ is concentrated in even degrees, then these obstructions must vanish. This in turn implies, by an inductive argument, that $A$ can be extended to a quantum $\ai$-algebra. Furthermore, if the cohomology vanishes entirely, it follows from Theorem \ref{thm_cohomact} that all the possible extensions at a given level must be equivalent, which implies by a similar inductive argument that this extension must be unique.
\end{proof}

\section{Examples} \label{sec_examples}

To provide an example of the construction outlined in Section \ref{sec_mc} and the obstruction theory described in Section \ref{sec_obsthy}, we turn to the family of one dimensional $\ai$-algebras that were introduced by Kontsevich in \cite{kontfeynman}. These are $\ai$-structures on a fixed vector space $W:=\gf$ of dimension one. Note that $V:=\Pi W$ is not a symplectic vector space according to the definition of Section \ref{sec_intro}, as it comes equipped with an \emph{even} inner product and not an \emph{odd} one. However, this situation can be treated in a manner analogous to the preceding sections, the only difference being that now we must consider the homology of the moduli space of curves \emph{with twisted coefficients}.

A cyclic $\ai$-structure on $W$ is an odd nilpotent hamiltonian $h\in\mathfrak{h}_{\geq 2}[V]$. It turns out that any odd hamiltonian is nilpotent and therefore determines a cyclic $\ai$-structure on $W$. This is due to the fact that any element of even order must vanish in $\mathfrak{h}_{\geq 2}[V]$, which in turn follows from the fact that $\mathfrak{h}[V]$ is the quotient of the free power series algebra in one odd variable $t$ by the action of the cyclic groups:
\[z_n\cdot t^n\mapsto\sgn(z_n)t^n=(-1)^{n-1}t^{n}.\]
Due to this fact, the bracket on $\mathfrak{h}[V]$ is trivial.

One can compute that
\begin{equation} \label{eqn_exmobs}
\Delta(t^{2n+1})=(2n+1)\nu t^{2n-1},
\end{equation}
hence $h$ does not satisfy the additional constraint of \eqref{eqn_quantumconstraint}. If, however, we work with the differential graded Lie algebra $\Lambda_{\gamma}[V]$, i.e. we work modulo the odd variable $\nu$, then \eqref{eqn_exmobs} vanishes and it follows that $h$ extends trivially to a solution of the master equation in $\Lambda_{\gamma}[V]$. Consequently, by \eqref{eqn_liftclass}, it produces a class in $\kcmp$ lifting the class defined in $\curv^{pt}$ by $h\in\mathfrak{h}_{\geq 2}[V]$.

To apply our obstruction theory to the problem, we must compute the cyclic Hochschild cohomology of this $\ai$-algebra. Since the cyclic Hochschild complex
\[(\mathfrak{h}_{\geq 1}[V],\ad h)\]
consists of just all functions $f\in\mathfrak{h}_{\geq 1}[V]$ of odd order, as all the even ones vanish, it follows that the differential is zero and the Hochschild cohomology is just the complex of all functions of odd order.

If we consider the problem of extending $h$ to a solution to the master equation in $\Lambda_{\gamma,\nu}[V]$, the first obstruction is given by a sum of terms of the form \eqref{eqn_exmobs}. Our calculation of the Hochschild cohomology implies that this is a nontrivial cocycle, hence $h$ cannot be extended to a solution to the master equation in $\Lambda_{\gamma,\nu}[V]$ (unless $h=0$).

Now we consider the problem of extending $h$ to a solution to the master equation in $\Lambda_{\gamma}[V]$, i.e. we work modulo $\nu$. It follows from our calculation of the Hochschild cohomology and (the correct corresponding analogue of) Corollary \ref{cor_vanish}, that all the obstructions must vanish, since they must lie in the degrees where the cohomology is zero. Furthermore, it follows that the equivalence classes of all possible extensions to the master equation at a given level are parameterised by products of odd functions. Hence we can obtain a new solution to the quantum master equation by adding on arbitrary products of odd functions; i.e. the general solution to the quantum master equation in $\Lambda_{\gamma}[V]$ is
\[\sum_{i=0}^\infty\sum_{k=1}^\infty\sum_{r_1,\ldots,r_k=0}^\infty a_{r_1,\ldots,r_k}^i\gamma^i (t^{2r_1+1}\wedge\cdots\wedge t^{2r_k+1})\]
from which we can obtain specific solutions by making arbitrary choices for the values of the coefficients $a_{r_1,\ldots,r_k}^i$.

This family of solutions to the quantum master equation extends the family of solutions to the classical master equation which were provided by Kontsevich in \cite{kontfeynman}. It is known that using the construction in Section \ref{sec_mc}, Kontsevich's solutions produce all the kappa classes in the cohomology of the open moduli space \cite{igusa}, \cite{mondello}, \cite{kontfeynman}. We have shown above that these classes can be lifted to the compactification $\kcmp$ of the moduli space defined by Kontsevich in \cite{kontairy}. This is in general agreement with the fact that the kappa classes are supported on the Deligne-Mumford compactification.

\end{document}